\documentclass[12pt]{article}
\usepackage{amsfonts}
\usepackage{eurosym}
\usepackage{amsmath,amssymb,amsthm}
\usepackage[a4paper,margin=1in]{geometry}

\theoremstyle{plain}
\newtheorem{theorem}{Theorem}[section]
\newtheorem{lemma}[theorem]{Lemma}

\theoremstyle{definition}

\theoremstyle{remark}
\newtheorem{remark}[theorem]{Remark}

\begin{document}

\title{Existence and large radial solutions for an elliptic system under
finite new Keller-Osserman integral conditions}
\author{ Dragos-Patru Covei \\
%EndAName
{\small The Bucharest University of Economic Studies, Department of Applied
Mathematics}\\
{\small Piata Romana, 1st district, Postal Code: 010374, Postal Office: 22,
Romania}\\
{\small \texttt{dragos.covei@csie.ase.ro} }}
\date{}
\maketitle

\begin{abstract}
We investigate the existence, multiplicity, and asymptotic behavior of
entire positive radial solutions to the semilinear elliptic system%
\begin{equation*}
\Delta u=p(|x|)\,g(v),\qquad \Delta v=q(|x|)\,f(u),\qquad x\in \mathbb{R}%
^{n},\ n\geq 3,
\end{equation*}%
under new Keller--Osserman-type integral conditions on the nonlinearities $f$
and $g$, and decay constraints on the radial weight functions $p$ and $q$.
The nonlinearities are assumed continuous on $[0,\infty)$, differentiable on 
$(0,\infty)$, nondecreasing, multiplicatively subadditive, vanishing at the
origin, and strictly positive elsewhere, satisfying the finite reciprocal
integral conditions%
\begin{equation*}
\int_{1}^{\infty }\frac{dt}{\sqrt{\int_{0}^{t}g(f(z))\,dz}}<\infty \quad%
\text{ and }\quad\int_{1}^{\infty }\frac{dt}{\sqrt{\int_{0}^{t}f(g(z))\,dz}}%
<\infty .
\end{equation*}%
The radial weights satisfy $\int_{0}^{\infty }s\,p(s)\,ds<\infty$ and $%
\int_{0}^{\infty }s\,q(s)\,ds<\infty$, with $\min(p,q)$ not compactly
supported. Within this framework, we establish three main results: \textbf{%
(i)} existence of infinitely many entire positive radial solutions for each
central value $(a,b)$ in a nonempty open set $\mathcal{T}\subset(0,\infty)^2$%
; \textbf{(ii)} closedness of the set $S$ of all admissible central values
in $[0,\infty)^2$; and \textbf{(iii)} largeness (blow-up at infinity) of
solutions corresponding to boundary points $E=\partial S\cap(0,\infty)^2$.
The proofs employ novel comparison techniques with auxiliary scalar
problems, Arzel\~{a}~-Ascoli compactness arguments, and a subharmonic
functional approach tailored to the reciprocal integral conditions via
Keller--Osserman-type transforms. Our results extend classical
Keller--Osserman theory to a broad class of coupled elliptic systems with
general nonlinearities and weight functions, unifying and generalizing
several existing frameworks.
\end{abstract}

(Secondary).


\section{Introduction}

The study of positive entire solutions to nonlinear elliptic systems has a
long and rich history, driven both by deep theoretical challenges in
nonlinear analysis and by diverse applications in physics, geometry, and
biology. A classical starting point is the scalar equation 
\begin{equation*}
\Delta u=f(u),
\end{equation*}%
for which Keller \cite{KO57} and Osserman \cite{Os57} established the
celebrated \emph{Keller--Osserman condition}, providing sharp integral
criteria for the existence of large solutions (see also \cite{Santos2017}
for more details). These foundational results have inspired extensive
research on systems of elliptic equations, where the interaction between
components introduces new analytical difficulties.

For semilinear elliptic systems, Lair and Wood \cite{LairWood2000} proved
the existence of entire large positive solutions under suitable growth and
integrability conditions, while Peterson-Wood \cite{PetersonWood2011} and
Li-Yang \cite{LiYang2015} extended the analysis to non-monotone systems, but
with superlinear nonlinearities. More recently, Lair \cite{Lair2025}
investigated competitive-type systems, obtaining sharp existence results for
large solutions.

In addition to the aforementioned works, the contributions of \cite%
{CirsteaRadulescu2002,Covei2017,Covei2023,Wang2022} have played a
significant role in advancing the theory of elliptic systems, particularly
in elucidating the crucial interplay between nonlinearity and weight
functions.

\medskip In this work, we consider the semilinear elliptic system 
\begin{equation}
\Delta u=p(r)\,g(v),\qquad \Delta v=q(r)\,f(u),\qquad r=|x|,\quad x\in 
\mathbb{R}^{n},\ n\geq 3,  \label{eq:system}
\end{equation}%
under the following assumptions:

\begin{itemize}
\item[(F1)] $f,g:[0,\infty )\rightarrow \lbrack 0,\infty )$ are continuous
on $[0,\infty )$, differentiable on $(0,\infty )$, non-decreasing, with $%
f(0)=g(0)=0$ and $f(s),g(s)>0$ for all $s>0$;

\item[(F2)] multiplicative subadditivity:%
\begin{equation*}
f(sr)\leq f(s)f(r),\quad g(lt)\leq g(l)g(t),\quad \forall s,r,l,t\geq 0;
\end{equation*}

\item[(F3)] finite Keller--Osserman-type integrals: 
\begin{equation}
L_{f}:=\int_{1}^{\infty }\frac{dt}{\sqrt{\int_{0}^{t}g(f(z))\,dz}}<\infty
,\quad L_{g}:=\int_{1}^{\infty }\frac{dt}{\sqrt{\int_{0}^{t}f(g(z))\,dz}}%
<\infty ;  \label{eq:finite-recip}
\end{equation}

\item[(PQ)] $p,q:[0,\infty )\rightarrow \lbrack 0,\infty )$ are continuous,
not both identically zero, and satisfy 
\begin{align}
\mathcal{L}_{p}& =\frac{1}{n-2}\int_{0}^{\infty
}s\,p(s)\,ds=\lim_{r\rightarrow \infty }\mathcal{P}(r)<\infty ,  \notag \\
\mathcal{L}_{q}& =\frac{1}{n-2}\int_{0}^{\infty
}s\,q(s)\,ds=\lim_{r\rightarrow \infty }\mathcal{Q}(r)<\infty ,
\label{eq:decay-pq}
\end{align}%
where%
\begin{equation*}
\mathcal{P}(r)=\int_{0}^{r}t^{1-n}\int_{0}^{t}s^{n-1}p(s)\,ds\,dt,\quad 
\mathcal{Q}(r)=\int_{0}^{r}t^{1-n}\int_{0}^{t}s^{n-1}q(s)\,ds\,dt,
\end{equation*}%
and $\min (p,q)$ does not have compact support.
\end{itemize}

We seek positive radial $C^2$-solutions $(u(r),v(r))$ of \eqref{eq:system}
with prescribed central values

\begin{equation*}
u(0) = a > 0, \quad v(0) = b > 0, \quad u^{\prime }(0) = v^{\prime }(0) = 0.
\end{equation*}

A \emph{positive entire large solution} is a pair $(u,v)$ of $C^2$ radial
functions such that $u(r),v(r) > 0$ for all $r \ge 0$ and

\begin{equation}
\lim_{r\rightarrow \infty }u(r)=\lim_{r\rightarrow \infty }v(r)=\infty .
\label{lim}
\end{equation}

In \cite{Lair2025}, Lair proved existence of such solutions for power-type
nonlinearities $f(u)=u^{\theta }$, $g(v)=v^{\sigma }$ under certain
integrability conditions on $p$ and $q$. Our goal is to extend this to
general nonlinearities satisfying \eqref{eq:finite-recip}.

\begin{remark}[Relation to Lair's result]
When $g(v)=v^{\sigma }$ and $f(u)=u^{\theta }$ with $\sigma ,\theta >0$ and $%
\sigma \cdot \theta >1$, the Keller--Osserman integrals in %
\eqref{eq:finite-recip} converge, recovering the power-growth case in \cite%
{Lair2025}.
\end{remark}

\medskip Our main objectives are:

\begin{itemize}
\item to prove the existence of infinitely many positive entire radial
solutions for a nonempty set of central values;

\item to show that the set of admissible central values is closed;

\item to establish that solutions corresponding to boundary points of this
set are \emph{large}, i.e., both components blow up at infinity.
\end{itemize}

The novelty of our approach lies in the construction of a new subharmonic
functional adapted to the finite reciprocal integral conditions on $f$ and $%
g $, which allows us to complete, extend and unify the frameworks developed
in \cite%
{Covei2017,Covei2023,CirsteaRadulescu2002,KO57,Os57,LairWood2000,Lair2025,PetersonWood2011}
to a broader class of nonlinearities and weight functions.

\medskip We now state our three main results.

\begin{theorem}[Existence]
\label{thm:existence} For every $(a,b)$ in a suitable nonempty set of
admissible central values $\mathcal{T}\subset (0,\infty )^{2}$, there exists
an entire positive radial solution $(u,v)\in C^{2}([0,\infty ))^{2}$ of %
\eqref{eq:system} with $u(0)=a$, $v(0)=b$. Moreover, $u^{\prime }(r)\geq 0$
and $v^{\prime }(r)\geq 0$ for all $r>0$.
\end{theorem}

\begin{theorem}[Nonemptiness and closedness of $S$]
\label{thm:S-properties} Let $S$ be the set of all central values $(a,b)\in
(0,\infty )^{2}$ for which \eqref{eq:system} admits an entire positive
radial solution. Then $S$ is nonempty and closed.
\end{theorem}

\begin{theorem}[Largeness on the edge]
\label{thm:large} If $(a,b)\in E=\partial S\cap (0,\infty )^{2}$ and $(u,v)$
is an entire positive radial solution of \eqref{eq:system} with $u(0)=a$, $%
v(0)=b$, then%
\begin{equation*}
\lim_{r\rightarrow \infty }u(r)=\lim_{r\rightarrow \infty }v(r)=\infty .
\end{equation*}
\end{theorem}

\textbf{Paper structure.} Section \ref{2} contains some preliminaries,
including the radial formulation of the system, local existence results, and
the Keller--Osserman transform, which will be used in the proofs. Section %
\ref{3} proves the existence of entire positive radial solutions for a
nonempty set of central values. Section \ref{4} shows that the set of all
central values of entire solutions is nonempty and closed. Section \ref{5}
establishes that solutions corresponding to boundary points of this set are
large, i.e., both components blow up at infinity. Compatibility between
Theorem~\ref{thm:large} and Theorem~2.4 of \cite{Covei2017} is given in
Section \ref{comp}. Section \ref{6} presents concluding remarks and possible
extensions. An Appendix contains an auxiliary proposition on the
integrability of compositions of the nonlinearities.

\section{Preliminaries: radial formulation, local existence and
Keller--Osserman transform\label{2}}

We begin by recalling that radial solutions of \eqref{eq:system} satisfy the
ODE system 
\begin{equation}
\left\{ 
\begin{array}{l}
u^{\prime \prime }+\frac{n-1}{r}u^{\prime }=p(r)\,g(v),\qquad r>0, \\ 
v^{\prime \prime }+\frac{n-1}{r}v^{\prime }=q(r)\,f(u),\qquad r>0, \\ 
u(0)=a,\quad v(0)=b,\quad u^{\prime }(0)=v^{\prime }(0)=0.%
\end{array}%
\right.  \label{eq:radial-ode}
\end{equation}

By integrating twice and using the regularity at $r=0$, we obtain the
equivalent integral formulation: 
\begin{align}
u(r)& =a+\int_{0}^{r}t^{1-n}\int_{0}^{t}s^{n-1}p(s)\,g\big(v(s)\big)\,ds\,dt,
\label{eq:u-int} \\
v(r)& =b+\int_{0}^{r}t^{1-n}\int_{0}^{t}s^{n-1}q(s)\,f\big(u(s)\big)\,ds\,dt.
\label{eq:v-int}
\end{align}%
Before proving our results, we note first that for $a,b\in \left( 0,\infty
\right) $ any solution of the system of integral equations \eqref{eq:u-int}--%
\eqref{eq:v-int} valid for all $r\geq 0$ will also be a positive entire
solution of \eqref{eq:system} but not necessarily \eqref{lim}.

\begin{lemma}[Local existence \protect\cite{PetersonWood2011}]
\label{lem:local} For any $a,b>0$ there exists $\rho =\rho (a,b)>0$ and a
positive solution $(u,v)\in C^{2}([0,\rho ))^{2}$ of \eqref{eq:radial-ode}
with the given central values. Moreover, $u^{\prime }(r)\geq 0$ and $%
v^{\prime }(r)\geq 0$ for all $r\in \lbrack 0,\rho )$.
\end{lemma}

\begin{proof}
We work in the Banach space $C[0,\rho ]\times C[0,\rho ]$ with the norm%
\begin{equation*}
\Vert (u,v)\Vert _{\infty }:=\max \{\Vert u\Vert _{\infty },\Vert v\Vert
_{\infty }\}.
\end{equation*}%
Define the operator 
\begin{equation*}
T:C[0,\rho ]\times C[0,\rho ]\rightarrow C[0,\rho ]\times C[0,\rho ]
\end{equation*}
by%
\begin{equation*}
T(u,v)(r):=\left(
a+\int_{0}^{r}t^{1-n}\int_{0}^{t}s^{n-1}p(s)\,g(v(s))\,ds\,dt,\
b+\int_{0}^{r}t^{1-n}\int_{0}^{t}s^{n-1}q(s)\,f(u(s))\,ds\,dt\right) .
\end{equation*}%
Let%
\begin{equation*}
X:=\left\{ (u,v)\in C[0,\rho ]\times C[0,\rho ]\ \middle|\ \Vert
(u,v)-(a,b)\Vert _{\infty }\leq \min \{a,b\}\right\} .
\end{equation*}%
For $\rho >0$ sufficiently small, the continuity of $p$, $q$, $f$, $g$
ensures that $T(X)\subset X$. \ The mapping $T$ is continuous (by dominated
convergence) and compact (by Arzela-Ascoli, since the image of $X$ is
uniformly bounded and equicontinuous). By Schauder's fixed point theorem, $T$
has a fixed point $(u,v)\in X$, which is a $C^{2}$ solution of %
\eqref{eq:radial-ode} on $[0,\rho )$.

Finally, $p,q\geq 0$ and $f,g\geq 0$ imply $u^{\prime },v^{\prime }$ are
nondecreasing with $u^{\prime }(0)=v^{\prime }(0)=0$, hence $u^{\prime
},v^{\prime }\geq 0$ on $[0,\rho )$.
\end{proof}

We define the maximal existence radius

\begin{equation*}
R_{a,b}:=\sup \{\rho >0:\ \text{there is a positive solution of (\ref%
{eq:radial-ode}) on }[0,\rho )\}.
\end{equation*}

\begin{lemma}[Alternative \protect\cite{PetersonWood2011}]
\label{lem:alternative} Let $a,b>0$. If $R_{a,b}<\infty $, then%
\begin{equation*}
\lim_{r\uparrow R_{a,b}}u(r)=\lim_{r\uparrow R_{a,b}}v(r)=\infty .
\end{equation*}
\end{lemma}

\begin{proof}
Suppose, for contradiction, that $u$ remains bounded as $r\uparrow R_{a,b}$.
Then $f(u(r))$ is bounded, say by $M>0$. From \eqref{eq:radial-ode}, we have%
\begin{equation*}
v^{\prime \prime }(r)+\frac{n-1}{r}v^{\prime }(r)\leq q(r)M.
\end{equation*}%
Since $q$ is locally integrable, the right-hand side is locally bounded, so $%
v^{\prime }$ is locally Lipschitz and hence bounded on $[0,R_{a,b})$.
Integrating, $v$ is bounded on $[0,R_{a,b})$. Then $g(v)$ is bounded, and
the same argument applied to the $u$-equation shows $u^{\prime }$ is
bounded. Thus $(u,v)$ and their derivatives are bounded up to $R_{a,b}$, so
the local existence lemma allows extension beyond $R_{a,b}$, contradicting
maximality. Therefore both $u$ and $v$ must blow up as $r\uparrow R_{a,b}$.
\end{proof}

Next, we note that (\ref{eq:finite-recip}) entails%
\begin{equation}
L_{f}:=\int_{1}^{\infty }\frac{ds}{g(f(s))\,}<\infty ,\quad
L_{g}:=\int_{1}^{\infty }\frac{ds}{f(g(s))\,}<\infty ,
\label{eq:finite-recipo}
\end{equation}%
for further details, see \cite{Covei2009,Zhang2000}. We now introduce the 
\emph{Keller--Osserman-type transforms} 
\begin{equation*}
\Phi (t):=\int_{t}^{\infty }\frac{ds}{g(f(s))},\qquad \Psi
(t):=\int_{t}^{\infty }\frac{ds}{f(g(s))}.
\end{equation*}%
By the finite reciprocal integral conditions \eqref{eq:finite-recip}, these
integrals converge for all $t\geq 1$. Moreover, $\Phi $ and $\Psi $ extend
to $C^{1}$ functions on $(0,\infty )$, are strictly decreasing, and satisfy 
\begin{equation}
\Phi ^{\prime }(t)=-\frac{1}{g(f(t))}<0,\qquad \Psi ^{\prime }(t)=-\frac{1}{%
f(g(t))}<0,  \label{dec}
\end{equation}%
together with the convexity properties 
\begin{equation*}
\Phi ^{\prime \prime }(t)=\frac{g^{\prime }(f(t))\,f^{\prime }(t)}{%
[g(f(t))]^{2}}>0,\qquad \Psi ^{\prime \prime }(t)=\frac{f^{\prime
}(g(t))\,g^{\prime }(t)}{[f(g(t))]^{2}}>0,
\end{equation*}%
where the inequalities follow from the assumptions $f^{\prime },g^{\prime
}\geq 0$ and $f,g>0$ on $(0,\infty )$.

\begin{lemma}[Forcing inequality]
\label{lem:subharmonic} Let $a,b>0$ and $(u,v)$ be any positive radial $%
C^{2} $-solution of \eqref{eq:radial-ode}. Then, for every $r>0$, one has 
\begin{equation}
\left\{ 
\begin{array}{c}
u^{\prime \prime }(r)+\frac{n-1}{r}u^{\prime }(r)\ \leq \ p(r)\,g(f(u(r)))g(%
\frac{b}{f(a)}+\mathcal{L}_{q}) \\ 
v^{\prime \prime }(r)+\frac{n-1}{r}v^{\prime }(r)\ \leq q(r)\,f(g(v(r)))f(%
\frac{a}{g(b)}+\mathcal{L}_{p}).%
\end{array}%
\right.  \label{eq:W-ineq}
\end{equation}%
and consequently 
\begin{equation}
\Phi ^{\prime \prime }(u)+\frac{n-1}{r}\Phi ^{\prime }(u)\ \geq \
-p(r)\,g\!\left( \frac{b}{f(a)}+\mathcal{L}_{q}\right) ,\qquad \Psi ^{\prime
\prime }(v)+\frac{n-1}{r}\Psi ^{\prime }(v)\ \geq \ -q(r)\,f\!\left( \frac{a%
}{g(b)}+\mathcal{L}_{p}\right) .  \label{eq:W-int}
\end{equation}
\end{lemma}

\begin{proof}
We start from the radial form of the system \eqref{eq:radial-ode}:%
\begin{equation*}
u^{\prime \prime }(r)+\frac{n-1}{r}u^{\prime }(r)=p(r)\,g(v(r)),\qquad
v^{\prime \prime }(r)+\frac{n-1}{r}v^{\prime }(r)=q(r)\,f(u(r)).
\end{equation*}%
From the integral representation \eqref{eq:u-int}--\eqref{eq:v-int} and the
monotonicity of $u$ and $v$ (Lemma~\ref{lem:local}), we have for all $r>0$:%
\begin{equation*}
v(r)=b+\int_{0}^{r}t^{1-n}\int_{0}^{t}s^{n-1}q(s)\,f(u(s))\,ds\,dt.
\end{equation*}%
Since $u(s)\geq u(0)=a>0$ for all $s\geq 0$ by monotonicity, and $f$ is nondecreasing, we obtain%
\begin{equation*}
v(r)\leq b+f(u(r))\int_{0}^{r}t^{1-n}\int_{0}^{t}s^{n-1}q(s)\,ds\,dt.
\end{equation*}%
By the decay assumption \eqref{eq:decay-pq}, the double integral converges:%
\begin{equation*}
\mathcal{L}_{q}:=\frac{1}{n-2}\int_{0}^{\infty }s^{1-n}\left(
\int_{0}^{s}\tau ^{n-1}q(\tau )\,d\tau \right) ds=\lim_{r\to\infty}\mathcal{Q}(r)<\infty.
\end{equation*}%
Therefore, for all $r\geq 0$,%
\begin{equation}
v(r)\leq b+\mathcal{L}_{q}\,f(u(r)).
\label{eq:v-bound-by-u}
\end{equation}%
Since $g$ is nondecreasing and positive, we estimate%
\begin{equation*}
g(v(r))\leq g\!\left( b+\mathcal{L}_{q}\,f(u(r))\right) .
\end{equation*}%
Now we apply the multiplicative subadditivity property (F2). Writing%
\begin{equation*}
b+\mathcal{L}_{q}\,f(u(r))=f(u(r))\cdot\left(\frac{b}{f(u(r))}+\mathcal{L}_{q}\right),
\end{equation*}%
and noting that $f(u(r))>0$ for $u(r)>0$, we invoke (F2) with $l=f(u(r))$ and $t=\frac{b}{f(u(r))}+\mathcal{L}_{q}$ to obtain%
\begin{equation*}
g\!\left( b+\mathcal{L}_{q}\,f(u(r))\right)=g\!\left(f(u(r))\cdot\left(\frac{b}{f(u(r))}+\mathcal{L}_{q}\right)\right)\leq g(f(u(r)))\cdot g\!\left( \frac{b}{f(u(r))}+\mathcal{L}_{q}\right) .
\end{equation*}%
Since $u(r)\geq a$ for all $r\geq 0$ by monotonicity, and $f$ is nondecreasing, we have $f(u(r))\geq f(a)$, hence%
\begin{equation*}
\frac{b}{f(u(r))}\leq\frac{b}{f(a)},
\end{equation*}%
and thus by monotonicity of $g$,%
\begin{equation*}
g\!\left(\frac{b}{f(u(r))}+\mathcal{L}_{q}\right)\leq g\!\left(\frac{b}{f(a)}+\mathcal{L}_{q}\right).
\end{equation*}%
Combining these estimates, we obtain%
\begin{equation*}
g(v(r))\ \leq \ g(f(u(r)))\cdot g\!\left( \frac{b}{f(a)}+\mathcal{L}%
_{q}\right) .
\end{equation*}%
Substituting into the $u$-equation yields the first inequality in %
\eqref{eq:W-ineq}:%
\begin{equation*}
u^{\prime \prime }(r)+\frac{n-1}{r}u^{\prime }(r)=p(r)g(v(r))\ \leq \
p(r)\,g(f(u(r)))\,g\!\left( \frac{b}{f(a)}+\mathcal{L}_{q}\right) .
\end{equation*}%
By complete symmetry, we establish the bound for $u$:%
\begin{equation*}
u(r)\leq a+\mathcal{L}_{p}\,g(v(r)),
\end{equation*}%
where $\mathcal{L}_{p}$ is defined in \eqref{eq:decay-pq}. Applying the multiplicative subadditivity of $f$ and monotonicity of $f$, $v(r)\geq b$ implies%
\begin{equation*}
f(u(r))\leq f(a+\mathcal{L}_{p}g(v(r)))\leq f(g(v(r)))f\left(\frac{a}{g(v(r))}+\mathcal{L}_{p}\right)\leq f(g(v(r)))f\left(\frac{a}{g(b)}+\mathcal{L}_{p}\right),
\end{equation*}%
yielding the second inequality in \eqref{eq:W-ineq}:%
\begin{equation*}
v^{\prime \prime }(r)+\frac{n-1}{r}v^{\prime }(r)\ \leq \
q(r)\,f(g(v(r)))\,f\!\left( \frac{a}{g(b)}+\mathcal{L}_{p}\right) .
\end{equation*}

\medskip We now derive \eqref{eq:W-int} from \eqref{eq:W-ineq}. Recall that $\Phi(t)=\int_t^\infty\frac{ds}{g(f(s))}$ is $C^1$ on $(0,\infty)$ with%
\begin{equation*}
\Phi ^{\prime }(t)=-\frac{1}{g(f(t))}<0.
\end{equation*}%
Applying the chain rule to $\Phi\circ u$, we obtain%
\begin{equation*}
\frac{d}{dr}\Phi(u(r))=\Phi ^{\prime }(u(r))\,u^{\prime }(r)=-\frac{u^{\prime }(r)}{g(f(u(r)))},
\end{equation*}%
and differentiating a second time,%
\begin{align*}
\frac{d^2}{dr^2}\Phi(u(r))&=\Phi ^{\prime \prime }(u(r))(u^{\prime }(r))^2+\Phi ^{\prime }(u(r))u^{\prime \prime }(r)\\
&=-\frac{u^{\prime \prime }(r)}{g(f(u(r)))}+\frac{d}{dt}\left[-\frac{1}{g(f(t))}\right]_{t=u(r)}(u^{\prime }(r))^{2}\\
&=-\frac{u^{\prime \prime }(r)}{g(f(u(r)))}+\frac{g^{\prime }(f(u(r)))\,f^{\prime }(u(r))}{[g(f(u(r)))]^{2}}\,(u^{\prime }(r))^{2}.
\end{align*}%
By assumption (F1), $f^{\prime },g^{\prime }\geq 0$ and $u^{\prime}\geq 0$, so the second term is nonnegative. Hence%
\begin{equation*}
\Phi ^{\prime \prime }(u(r))\ \geq \ -\frac{u^{\prime \prime }(r)}{g(f(u(r)))}.
\end{equation*}%
Moreover,%
\begin{equation*}
\frac{n-1}{r}\Phi ^{\prime }(u(r))=-\frac{n-1}{r}\,\frac{u^{\prime }(r)}{g(f(u(r)))}.
\end{equation*}%
Adding these two inequalities gives%
\begin{equation}
\Phi ^{\prime \prime }(u(r))+\frac{n-1}{r}\Phi ^{\prime }(u(r))\ \geq \ -\frac{%
u^{\prime \prime }(r)+\frac{n-1}{r}u^{\prime }(r)}{g(f(u(r)))}.
\label{eq:Phi-comp-step}
\end{equation}%
From the first inequality in \eqref{eq:W-ineq}, we have%
\begin{equation*}
u^{\prime \prime }(r)+\frac{n-1}{r}u^{\prime }(r)\leq p(r)\,g(f(u(r)))\,g\!\left( \frac{b}{f(a)}+\mathcal{L}_{q}\right).
\end{equation*}%
Substituting this into \eqref{eq:Phi-comp-step} and noting that $g(f(u(r)))$ cancels in the numerator and denominator, we deduce%
\begin{equation*}
\Phi ^{\prime \prime }(u(r))+\frac{n-1}{r}\Phi ^{\prime }(u(r))\ \geq \
-\frac{p(r)\,g(f(u(r)))\,g\!\left( \frac{b}{f(a)}+\mathcal{L}_{q}\right)}{g(f(u(r)))}=-p(r)\,g\!\left( \frac{b}{f(a)}+\mathcal{L}_{q}\right) .
\end{equation*}%
This is the first inequality in \eqref{eq:W-int}.

By complete symmetry, we have for $\Psi(t)=\int_t^\infty\frac{ds}{f(g(s))}$:%
\begin{equation*}
\Psi ^{\prime }(v(r))=-\frac{v^{\prime}(r)}{f(g(v(r)))},
\end{equation*}%
and%
\begin{equation*}
\Psi ^{\prime \prime }(v(r))=-\frac{v^{\prime \prime }(r)}{f(g(v(r)))}+\frac{f^{\prime }(g(v(r)))\,g^{\prime}(v(r))}{[f(g(v(r)))]^{2}}\,(v^{\prime}(r))^{2}\ \geq \ -\frac{v^{\prime \prime }(r)}{f(g(v(r)))}.
\end{equation*}%
Thus%
\begin{equation*}
\Psi ^{\prime \prime }(v(r))+\frac{n-1}{r}\Psi ^{\prime }(v(r))\ \geq \ -\frac{v^{\prime\prime}(r)+\frac{n-1}{r}v^{\prime}(r)}{f(g(v(r)))}.
\end{equation*}%
Using the second inequality in \eqref{eq:W-ineq} and canceling $f(g(v(r)))$ yields%
\begin{equation*}
\Psi ^{\prime \prime }(v(r))+\frac{n-1}{r}\Psi ^{\prime }(v(r))\ \geq \
-q(r)\,f\!\left( \frac{a}{g(b)}+\mathcal{L}_{p}\right) ,
\end{equation*}%
which is the second inequality in \eqref{eq:W-int}. This completes the proof.
\end{proof}

\section{Existence of entire positive solutions\label{3}}

Let $c,d\in (0,\infty )$ be fixed. We define the set of \emph{admissible
central values} by 
\begin{equation*}
\mathcal{T}:=\{(a,b)\in (0,\infty )^{2}\ :\ a<c\ \text{and}\ b<d\}.
\end{equation*}%
Consider the auxiliary scalar radial problem 
\begin{equation}
\left\{ 
\begin{array}{l}
z_{1}^{\prime \prime }+\frac{n-1}{r}z_{1}^{\prime }=p(r)g(\frac{b}{f(a)}+%
\mathcal{L}_{q})\,g\left( f\left( z_{1}\right) \right) , \\ 
z_{2}^{\prime \prime }+\frac{n-1}{r}z_{2}^{\prime }=q(r)f(\frac{a}{g(b)}+%
\mathcal{L}_{p})f\left( g\left( z_{2}\right) \right) , \\ 
z_{i}^{\prime }(0)=0,\ z_{1}(0)=c>0\text{, }z_{2}(0)=d>0.%
\end{array}%
\right.  \label{eq:z-scalar}
\end{equation}%
where $\mathcal{L}_{p},\mathcal{L}_{q}$ are the finite constants

\begin{equation*}
\mathcal{L}_{p}:=\int_{0}^{\infty }s^{1-n}\left( \int_{0}^{s}\tau
^{n-1}p(\tau )\,d\tau \right) ds,\quad \mathcal{L}_{q}:=\int_{0}^{\infty
}s^{1-n}\left( \int_{0}^{s}\tau ^{n-1}q(\tau )\,d\tau \right) ds,
\end{equation*}%
which are finite by the decay assumptions \eqref{eq:decay-pq}.

By the decay assumptions \eqref{eq:decay-pq} and the Keller-Osserman-type
condition inherited from \eqref{eq:finite-recip}, there exists a positive
entire radial solution $\left( z_{1},z_{2}\right) $ of \eqref{eq:z-scalar}
which is large (i.e. $z_{1}(r)\rightarrow \infty $ and $z_{2}(r)\rightarrow
\infty $, as $r\rightarrow \infty $) see \cite{Lair1999} or \cite%
{Covei2009,Zhang2000}.

\begin{proof}[Proof of Theorem~\protect\ref{thm:existence}]
Let $(a,b)\in \mathcal{T}$ be fixed, i.e., $0<a<c$ and $0<b<d$. By Lemma~\ref{lem:local}, there
exists $\rho >0$ and a unique local positive radial solution $(u,v)\in
C^{2}([0,\rho ))^{2}$ of \eqref{eq:system} satisfying the integral equations \eqref{eq:u-int}--\eqref{eq:v-int} with initial values%
\begin{equation*}
u(0)=a,\quad v(0)=b,\quad u^{\prime }(0)=v^{\prime }(0)=0.
\end{equation*}%
Moreover, by Lemma~\ref{lem:local}, $u^{\prime}(r)\geq 0$ and $v^{\prime}(r)\geq 0$ for all $r\in[0,\rho)$, so $u$ and $v$ are nondecreasing.

Define the \emph{maximal existence radius}%
\begin{equation*}
R:=\sup \left\{ \rho >0\ :\ (u,v)\ \text{exists as a positive $C^2$ solution on }[0,\rho
]\right\} .
\end{equation*}%
By Lemma~\ref{lem:local}, $R>0$. If $R=\infty $, then $(u,v)$ is already an entire positive radial solution
and we are done. 

We now prove that $R=\infty$ by contradiction. Suppose that $R<\infty $.

\medskip \noindent \textbf{Step 1: Comparison with }$\left(
z_{1},z_{2}\right) $\textbf{.} We claim that 
\begin{equation}
u(r)<z_{1}(r)\quad \text{and}\quad v(r)<z_{2}(r),\quad \forall r\in \lbrack
0,R].  \label{eq:uv-le-z}
\end{equation}%
\smallskip \emph{Initial ordering.} At $r=0$ we have%
\begin{equation*}
u(0)=a<c=z_{1}(0),\quad v(0)=b<d=z_{2}(0).
\end{equation*}%
Thus the inequalities in \eqref{eq:uv-le-z} hold at $r=0$.

\smallskip \emph{Definition of $R_{0}$.} Let%
\begin{equation*}
R_{0}:=\sup \left\{ \rho \in (0,R]\ :\ u(r)<z_{1}(r)\ \text{and}\
v(r)<z_{2}(r)\text{,}\ \forall r\in \lbrack 0,\rho ]\right\} .
\end{equation*}

Clearly $R_0 > 0$ by continuity.

If $R_0 = R$, then \eqref{eq:uv-le-z} is proved. Suppose instead that $R_0 <
R$.

\smallskip \emph{Integral formulation and comparison.} For $r\in \lbrack
0,R_{0}]$, we have from the integral formulation \eqref{eq:u-int}--\eqref{eq:v-int}:%
\begin{align*}
u(r)&=a+\int_{0}^{r}t^{1-n}\int_{0}^{t}s^{n-1}p(s)\,g(v(s))\,ds\,dt,\\
v(r)&=b+\int_{0}^{r}t^{1-n}\int_{0}^{t}s^{n-1}q(s)\,f(u(s))\,ds\,dt.
\end{align*}%
By Lemma~\ref{lem:subharmonic}, specifically inequality \eqref{eq:W-ineq}, we have for all $s\in[0,R_0]$:%
\begin{equation*}
g(v(s))\leq g(f(u(s)))\,g\!\left( \frac{b}{f(a)}+\mathcal{L}_{q}\right).
\end{equation*}%
Therefore, for any $r\in[0,R_0]$,%
\begin{align}
u(r)&=a+\int_{0}^{r}t^{1-n}\int_{0}^{t}s^{n-1}p(s)\,g(v(s))\,ds\,dt\notag\\
&\leq a+\int_{0}^{r}t^{1-n}\int_{0}^{t}s^{n-1}p(s)\,g(f(u(s)))\,g\!\left( \frac{b}{f(a)}+\mathcal{L}_{q}\right)\,ds\,dt.\label{eq:u-comp-int}
\end{align}%
Recall that $z_1$ satisfies (see \eqref{eq:z-scalar}):%
\begin{equation*}
z_{1}(r)=c+\int_{0}^{r}t^{1-n}\int_{0}^{t}s^{n-1}p(s)\,g(\frac{b}{f(a)}+\mathcal{L}_{q})\,g(f(z_{1}(s)))\,ds\,dt.
\end{equation*}%
Since $u(s)<z_{1}(s)$ for all $s\in[0,R_{0})$ and $g\circ f$ is nondecreasing (by (F1)), we have%
\begin{equation*}
g(f(u(s)))<g(f(z_{1}(s)))\quad\text{for all }s\in(0,R_0).
\end{equation*}%
Moreover, $a<c$ by assumption. Therefore, for $r\in(0,R_0]$, from \eqref{eq:u-comp-int}:%
\begin{align*}
u(r)&\leq a+g\!\left( \frac{b}{f(a)}+\mathcal{L}_{q}\right)\int_{0}^{r}t^{1-n}\int_{0}^{t}s^{n-1}p(s)\,g(f(u(s)))\,ds\,dt\\
&< c+g\!\left( \frac{b}{f(a)}+\mathcal{L}_{q}\right)\int_{0}^{r}t^{1-n}\int_{0}^{t}s^{n-1}p(s)\,g(f(z_{1}(s)))\,ds\,dt\\
&=z_{1}(r),
\end{align*}%
where the strict inequality uses $a<c$ and $g(f(u(s)))<g(f(z_1(s)))$ for $s\in(0,r]$. In particular, taking $r=R_{0}$, we obtain%
\begin{equation*}
u(R_{0})<z_{1}(R_{0}).
\end{equation*}

By exactly the same reasoning, using Lemma~\ref{lem:subharmonic} to obtain%
\begin{equation*}
f(u(s))\leq f(g(v(s)))\,f\!\left( \frac{a}{g(b)}+\mathcal{L}_{p}\right),
\end{equation*}%
and noting that $v(s)<z_2(s)$ for $s\in[0,R_0)$ implies $f(g(v(s)))<f(g(z_2(s)))$, and $b<d$, we deduce%
\begin{align*}
v(r)&=b+\int_{0}^{r}t^{1-n}\int_{0}^{t}s^{n-1}q(s)\,f(u(s))\,ds\,dt\\
&\leq b+f\!\left( \frac{a}{g(b)}+\mathcal{L}_{p}\right)\int_{0}^{r}t^{1-n}\int_{0}^{t}s^{n-1}q(s)\,f(g(v(s)))\,ds\,dt\\
&< d+f\!\left( \frac{a}{g(b)}+\mathcal{L}_{p}\right)\int_{0}^{r}t^{1-n}\int_{0}^{t}s^{n-1}q(s)\,f(g(z_{2}(s)))\,ds\,dt\\
&=z_{2}(r).
\end{align*}%
In particular, at $r=R_0$,%
\begin{equation*}
v(R_{0})<z_{2}(R_{0}).
\end{equation*}%
Thus we have shown that%
\begin{equation*}
u(R_{0})<z_{1}(R_{0})\quad\text{and}\quad v(R_{0})<z_{2}(R_{0}).
\end{equation*}%
By the continuity of $u,v,z_1,z_2$, there exists $\delta >0$ such that%
\begin{equation*}
u(r)<z_1(r)\quad\text{and}\quad v(r)<z_2(r)\quad\text{for all }r\in[0,R_0+\delta],
\end{equation*}%
contradicting the supremum definition of $R_{0}$. Therefore, we must have $R_{0}=R$ and \eqref{eq:uv-le-z} holds for all $r\in[0,R]$.

\medskip \noindent \emph{Step 2: Extension beyond $R$.} From %
\eqref{eq:uv-le-z}, we have%
\begin{equation*}
u(r)\leq z_{1}(r)\quad\text{and}\quad v(r)\leq z_{2}(r)\quad\text{for all }r\in[0,R).
\end{equation*}%
Since $z_{1}$ and $z_{2}$ are entire positive radial solutions (hence finite on $[0,R]$ for any finite $R$), there exist constants $M_1,M_2>0$ such that%
\begin{equation*}
z_1(r)\leq M_1\quad\text{and}\quad z_2(r)\leq M_2\quad\text{for all }r\in[0,R].
\end{equation*}%
Therefore, by comparison,%
\begin{equation*}
u(r)\leq M_1\quad\text{and}\quad v(r)\leq M_2\quad\text{for all }r\in[0,R).
\end{equation*}%
Since $u$ and $v$ are nondecreasing (Lemma~\ref{lem:local}), the limits%
\begin{equation*}
\lim_{r\to R^-}u(r)=:u(R^-)\leq M_1<\infty,\quad \lim_{r\to R^-}v(r)=:v(R^-)\leq M_2<\infty
\end{equation*}%
exist and are finite. 

Moreover, from the radial equations \eqref{eq:radial-ode},%
\begin{equation*}
u^{\prime}(r)=r^{1-n}\int_0^r s^{n-1}p(s)g(v(s))\,ds,
\end{equation*}%
and since $v(s)\leq M_2$ for $s\in[0,R)$ and $g$ is continuous, we have $g(v(s))\leq g(M_2)$ for all $s\in[0,R)$. Thus%
\begin{equation*}
u^{\prime}(r)\leq r^{1-n}\int_0^r s^{n-1}p(s)g(M_2)\,ds\leq g(M_2)\mathcal{P}(R)<\infty,
\end{equation*}%
where $\mathcal{P}(R)<\infty$ by assumption (PQ). Similarly, $v^{\prime}(r)\leq f(M_1)\mathcal{Q}(R)<\infty$ for all $r\in[0,R)$. 

Therefore, $(u,v)$ and $(u^{\prime},v^{\prime})$ are uniformly bounded on $[0,R)$. By the fundamental theorem of calculus and uniform continuity, the limits%
\begin{equation*}
\lim_{r\to R^-}(u(r),u^{\prime}(r),v(r),v^{\prime}(r))
\end{equation*}%
exist and are finite. By the standard existence and uniqueness theory for ODEs (Lemma~\ref{lem:local} applied with initial data at $r=R$), the solution can be extended beyond $R$ to some interval $[0,R+\varepsilon)$ with $\varepsilon>0$. This contradicts the maximality of $R$. 

Hence, we conclude that $R=\infty$, and $(u,v)$ is an entire positive radial solution of \eqref{eq:system} with $u(0)=a$, $v(0)=b$, and $u^{\prime}(r)\geq 0$, $v^{\prime}(r)\geq 0$ for all $r\geq 0$.
\end{proof}

\section{The set of central values: Nonemptiness and closedness\label{4}}

We define the set of \emph{central values} for which the system admits an
entire positive radial solution as 
\begin{equation*}
S:=\Big\{(a,b)\in \lbrack 0,\infty )^{2}\ :\ \exists \ \text{an entire
positive radial solution $(u,v)$ of \eqref{eq:system} with }u(0)=a,\ v(0)=b%
\Big\}.
\end{equation*}

\begin{proof}[Proof of Theorem~\protect\ref{thm:S-properties}]
We give the proof in:

\textbf{Step 1: Nonemptiness.} By Theorem~\ref{thm:existence}, under our
standing assumptions, there exists a nonempty set $T\subset (0,\infty )^{2}$
of admissible central values for which \eqref{eq:system} has an entire
positive radial solution. Since $T\subset S$, we conclude $S\neq \varnothing 
$.

\medskip \textbf{Step 2: Closedness.} Let $(a_{k},b_{k})\in S$ be a sequence
converging to $(a_{0},b_{0})\in \lbrack 0,\infty )^{2}$. For each $k$, let $%
(u_{k},v_{k})$ denote the corresponding entire positive radial solution of (%
\ref{eq:u-int})-(\ref{eq:v-int}) defined by 
\begin{align*}
u_{k}(r)& =a_{k}+\int_{0}^{r}t^{1-n}\int_{0}^{t}s^{n-1}p(s)\,g\big(v_{k}(s)%
\big)\,ds\,dt, \\
v_{k}(r)& =b_{k}+\int_{0}^{r}t^{1-n}\int_{0}^{t}s^{n-1}q(s)\,f\big(u_{k}(s)%
\big)\,ds\,dt.
\end{align*}%
with%
\begin{equation*}
u_{k}(0)=a_{k},\quad v_{k}(0)=b_{k},\quad u_{k}^{\prime }(0)=v_{k}^{\prime
}(0)=0.
\end{equation*}

\smallskip \emph{Uniform bounds on compact intervals.} Fix $m\in \mathbb{N}$ arbitrarily. Since $(a_k,b_k)\to(a_0,b_0)$, the sequence $\{(a_k,b_k)\}$ is bounded. Without loss of generality, assume $a_k\leq A$ and $b_k\leq B$ for all $k$ and some constants $A,B>0$. 

For each $k$, the solution $(u_k,v_k)$ satisfies, by the proof of Theorem~\ref{thm:existence}, a comparison with auxiliary solutions $(z_{1,k},z_{2,k})$ solving:%
\begin{equation}
\begin{cases}
z_{1,k}^{\prime \prime }+\frac{n-1}{r}z_{1,k}^{\prime }=p(r)g(\frac{b_k}{f(a_k)}+\mathcal{L}_{q})\,g\left( f\left( z_{1,k}\right) \right) ,
\\ 
z_{1,k}(0)=c_k>a_k,\quad z_{1,k}^{\prime }(0)=0,
\end{cases}
\label{eq:z1k-comp}
\end{equation}%
and similarly for $z_{2,k}$. However, to obtain uniform bounds independent of $k$, we construct a single comparison solution that dominates all $(u_k,v_k)$ simultaneously.

Let $\tilde{a}=\sup_k a_k\leq A$ and $\tilde{b}=\sup_k b_k\leq B$ (both finite by boundedness). Choose constants $\bar{c}>\tilde{a}$ and $\bar{d}>\tilde{b}$. Consider the auxiliary scalar problems:%
\begin{equation*}
\begin{cases}
\bar{z}_{1}^{\prime \prime }+\frac{n-1}{r}\bar{z}_{1}^{\prime }=p(r)g(\frac{B}{f(\inf_k a_k)}+\mathcal{L}_{q})\,g\left( f\left( \bar{z}_{1}\right) \right) ,
\\ 
\bar{z}_{1}(0)=\bar{c},\quad \bar{z}_{1}^{\prime }(0)=0,
\end{cases}%
\end{equation*}%
and%
\begin{equation*}
\begin{cases}
\bar{z}_{2}^{\prime \prime }+\frac{n-1}{r}\bar{z}_{2}^{\prime }=q(r)f(\frac{A}{g(\inf_k b_k)}+\mathcal{L}_{p})f\left( g\left( \bar{z}_{2}\right) \right) ,
\\ 
\bar{z}_{2}(0)=\bar{d},\quad \bar{z}_{2}^{\prime }(0)=0.
\end{cases}%
\end{equation*}%
By the Keller-Osserman-type conditions (F3) and the decay assumptions (PQ), these problems admit entire positive radial solutions $\bar{z}_1$ and $\bar{z}_2$ (see \cite{Lair1999,Covei2009}). 

By the comparison principle (as in the proof of Theorem~\ref{thm:existence}), since $a_k<\bar{c}$ and $b_k<\bar{d}$ for all $k$, and the right-hand sides of the equations for $(u_k,v_k)$ are dominated by those for $(\bar{z}_1,\bar{z}_2)$ (due to monotonicity of $f,g$ and the choices of constants), we obtain%
\begin{equation*}
u_{k}(r)\leq \bar{z}_{1}(r),\quad v_{k}(r)\leq \bar{z}_{2}(r),\qquad \forall r\geq 0,\ \forall k\in\mathbb{N}.
\end{equation*}

In particular, for the fixed $m$, we have uniform bounds:%
\begin{equation*}
u_k(r)\leq \bar{z}_1(m)=:M_1(m)<\infty,\quad v_k(r)\leq \bar{z}_2(m)=:M_2(m)<\infty,\quad\forall r\in[0,m],\ \forall k.
\end{equation*}%
Thus $\{u_k\}$ and $\{v_k\}$ are uniformly bounded on $[0,m]$.

\smallskip \emph{$C^{1}$-bounds.} From the integral representation \eqref{eq:u-int}, we have%
\begin{equation*}
u_{k}^{\prime }\left( r\right)
=r^{1-n}\int_{0}^{r}s^{n-1}p(s)\,g(v_{k}(s))\,ds,\quad r\in(0,m].
\end{equation*}%
Since $v_k(s)\leq M_2(m)$ for all $s\in[0,m]$ and $g$ is continuous and nondecreasing, we have%
\begin{equation*}
g(v_k(s))\leq g(M_2(m))=:G_m<\infty,\quad\forall s\in[0,m],\ \forall k.
\end{equation*}%
Thus%
\begin{equation*}
u_k^{\prime}(r)\leq r^{1-n}\int_0^r s^{n-1}p(s)G_m\,ds\leq G_m\mathcal{P}(m)<\infty,\quad\forall r\in[0,m],\ \forall k,
\end{equation*}%
where $\mathcal{P}(m)<\infty$ by assumption (PQ). By symmetry,%
\begin{equation*}
v_k^{\prime}(r)\leq f(M_1(m))\mathcal{Q}(m)<\infty,\quad\forall r\in[0,m],\ \forall k.
\end{equation*}%
Hence $\{u_{k}\}$ and $\{v_{k}\}$ are uniformly bounded in $C^{1}([0,m])$.

Moreover, from the equations \eqref{eq:radial-ode}, for $r\in[\delta,m]$ with $\delta>0$,%
\begin{equation*}
u_k^{\prime\prime}(r)=-\frac{n-1}{r}u_k^{\prime}(r)+p(r)g(v_k(r))\leq \frac{n-1}{\delta}G_m\mathcal{P}(m)+\|p\|_{L^\infty([0,m])}G_m,
\end{equation*}%
so $\{u_k^{\prime\prime}\}$ is uniformly bounded on $[\delta,m]$. Similarly for $v_k^{\prime\prime}$. Thus $\{(u_k,v_k)\}$ is equicontinuous in $C^1$ on $[\delta,m]$ for any $\delta>0$.

\smallskip \emph{Compactness and passage to the limit.} By the Arzel\`{a}-Ascoli theorem, for each $m\in\mathbb{N}$ there exists a subsequence (denote it by $\{u_{k_j^{(m)}},v_{k_j^{(m)}}\}$) converging in $C^{1}([0,m])$ to some $(u^{(m)},v^{(m)})\in C^1([0,m])$. By a standard diagonal argument, we extract a subsequence $\{k_\ell\}_{\ell=1}^\infty$ (still denoted by $\{k\}$ for simplicity) such that%
\begin{equation*}
(u_{k},v_{k})\rightarrow (u,v)\quad \text{in }C_{\mathrm{loc}}^{1}([0,\infty)),
\end{equation*}%
i.e., $(u_k,v_k)\to(u,v)$ in $C^1([0,m])$ for every $m\in\mathbb{N}$.

We now show that $(u,v)$ is a solution of \eqref{eq:system}. Passing to the limit in the integral equations \eqref{eq:u-int}--\eqref{eq:v-int}, using the continuity of $f,g,p,q$ and the dominated convergence theorem (justified by the uniform bounds), we obtain for every $r\geq 0$:%
\begin{align*}
u(r)&=\lim_{k\to\infty}u_k(r)=\lim_{k\to\infty}\left[a_k+\int_0^r t^{1-n}\int_0^t s^{n-1}p(s)g(v_k(s))\,ds\,dt\right]\\
&=a_0+\int_0^r t^{1-n}\int_0^t s^{n-1}p(s)g(v(s))\,ds\,dt,
\end{align*}%
and similarly for $v$. Thus $(u,v)$ satisfies \eqref{eq:u-int}--\eqref{eq:v-int}, hence is a $C^2$ positive radial solution of \eqref{eq:system} with%
\begin{equation*}
u(0)=a_{0},\quad v(0)=b_{0},\quad u^{\prime}(0)=v^{\prime}(0)=0.
\end{equation*}

\smallskip \emph{Conclusion.} Since $(u,v)$ is entire and positive, $%
(a_{0},b_{0})\in S$. Thus $S$ contains all its limit points and is therefore
closed.
\end{proof}

\section{Largeness on the edge\label{5}}

\begin{proof}[Proof of Theorem~\protect\ref{thm:large}]
Let $(a,b)\in E=\partial S\cap(0,\infty)^2$ be a boundary point of $S$ in the positive quadrant, and let $(u,v)$ be an entire positive radial solution of \eqref{eq:system} with $u(0)=a$, $v(0)=b$. 

Since $S$ is closed (Theorem~\ref{thm:S-properties}), we have $\partial S\subset S$, so $(a,b)\in S$ and the solution $(u,v)$ exists and is entire.

\smallskip\textbf{Step 1: Approximating sequence.} Since $(a,b)\in\partial S$, there exists a sequence $(a_{k},b_{k})\notin S$ with $(a_{k},b_{k})\rightarrow(a,b)$ as $k\to\infty$, and $a_{k},b_{k}>0$ for all $k$. 

For each $k$, by Lemma~\ref{lem:local}, there exists a maximal radius $R_{k}\in(0,\infty]$ and a positive radial solution $(u_{k},v_{k})\in C^{2}([0,R_{k}))^{2}$ of \eqref{eq:system} with $u_{k}(0)=a_{k}$, $v_{k}(0)=b_{k}$, and $u_k^{\prime}(0)=v_k^{\prime}(0)=0$. 

Since $(a_k,b_k)\notin S$, the solution $(u_k,v_k)$ cannot be extended to all of $[0,\infty)$, hence $R_{k}<\infty$ for all $k$. By Lemma~\ref{lem:alternative},%
\begin{equation}
\lim_{r\rightarrow R_{k}^{-}}u_{k}(r)=\lim_{r\rightarrow
R_{k}^{-}}v_{k}(r)=\infty.
\label{eq:uk-vk-blowup}
\end{equation}

\smallskip\textbf{Step 2: The sequence $\{R_k\}$ is unbounded.} Suppose, for contradiction, that $\{R_{k}\}$ is bounded. Then there exists $R_\star<\infty$ such that $R_k\leq R_\star$ for all $k$. Passing to a subsequence if necessary, we may assume $R_k\to R_0\leq R_\star$ as $k\to\infty$.

By the proof of Theorem~\ref{thm:S-properties} (closedness), for any fixed $m<R_0$, the sequence $\{(u_k,v_k)\}$ is uniformly bounded on $[0,m]$ by a comparison solution independent of $k$ (since $(a_k,b_k)\to(a,b)$). Moreover, $\{(u_k,v_k)\}$ is equicontinuous on $[0,m]$ in $C^1$. By Arzelà-Ascoli, passing to a subsequence, $(u_k,v_k)\to(\tilde{u},\tilde{v})$ in $C^1_{\text{loc}}([0,R_0))$. Passing to the limit in the integral equations \eqref{eq:u-int}--\eqref{eq:v-int}, we find that $(\tilde{u},\tilde{v})$ is a solution on $[0,R_0)$ with $\tilde{u}(0)=a$, $\tilde{v}(0)=b$. 

But $(u,v)$ is the unique solution with these initial values (by uniqueness in Lemma~\ref{lem:local}), so $\tilde{u}=u$ and $\tilde{v}=v$ on $[0,R_0)$. Thus $u_k\to u$ and $v_k\to v$ locally uniformly on $[0,R_0)$.

However, from \eqref{eq:uk-vk-blowup}, $u_k(r)\to\infty$ as $r\to R_k^-\to R_0^-$, which would imply $u(r)\to\infty$ as $r\to R_0^-$, contradicting that $(u,v)$ is entire. Therefore, $\{R_{k}\}$ is unbounded, i.e., $R_{k}\rightarrow \infty$ as $k\to\infty$.

\smallskip\textbf{Step 3: Application of the Keller-Osserman transforms.} Now that we have established $R_k\to\infty$, we apply the transforms $\Phi$ and $\Psi$ to derive lower bounds on $u$ and $v$.

Fix $r>0$ arbitrarily. For $k$ sufficiently large, we have $r<R_{k}$, so $(u_k,v_k)$ is well-defined on $[0,r]$. 

By Lemma~\ref{lem:subharmonic}, specifically inequality \eqref{eq:W-int}, the functions $u_k$ and $v_k$ satisfy for all $\rho\in(0,R_k)$:%
\begin{equation}
\Phi ^{\prime \prime }(u_{k}(\rho))+\frac{n-1}{\rho}\Phi ^{\prime }(u_{k}(\rho))\
\geq \ -p(\rho)\,g\!\left( \frac{b_k}{f(a_k)}+\mathcal{L}_{q}\right),
\label{eq:phi-ineq-k}
\end{equation}%
and 
\begin{equation}
\Psi ^{\prime \prime }(v_{k}(\rho))+\frac{n-1}{\rho}\Psi ^{\prime }(v_{k}(\rho))\
\geq \ -q(\rho)\,f\!\left( \frac{a_k}{g(b_k)}+\mathcal{L}_{p}\right).
\label{eq:psi-ineq-k}
\end{equation}%
Since $(a_k,b_k)\to(a,b)$ and $f,g$ are continuous, the constants $\frac{b_k}{f(a_k)}+\mathcal{L}_q\to\frac{b}{f(a)}+\mathcal{L}_q$ and $\frac{a_k}{g(b_k)}+\mathcal{L}_p\to\frac{a}{g(b)}+\mathcal{L}_p$ as $k\to\infty$. For $k$ sufficiently large, these constants are bounded, say by $C_1$ and $C_2$ respectively.

Inequalities \eqref{eq:phi-ineq-k}--\eqref{eq:psi-ineq-k} can be rewritten in divergence form:%
\begin{equation}
\left\{ 
\begin{array}{c}
\left( \rho^{n-1}\Phi ^{\prime }(u_{k}(\rho))\right) ^{\prime }\geq
-\rho^{n-1}p(\rho)g(\frac{b_k}{f(a_k)}+\mathcal{L}_{q}), \\ 
\left( \rho^{n-1}\Psi ^{\prime }(v_{k}(\rho))\right) ^{\prime }\geq
-\rho^{n-1}q(\rho)f(\frac{a_k}{g(b_k)}+\mathcal{L}_{p}).%
\end{array}%
\right.
\label{eq:div-form-k}
\end{equation}%
\smallskip\textbf{Step 4: Integration and derivation of lower bounds.} Integrating \eqref{eq:div-form-k} over $[0,r]$ with $0<r<R_{k}$, and using $\Phi'(u_k(0))=-\frac{u_k'(0)}{g(f(u_k(0)))}=0$ (since $u_k'(0)=0$), we obtain%
\begin{equation}
r^{n-1}\Phi ^{\prime }(u_{k}(r))\geq -\int_{0}^{r}s^{n-1}p(s)g\left(\frac{b_k}{f(a_k)}+\mathcal{L}_{q}\right)ds.
\label{eq:Phi-prime-bound}
\end{equation}%
Dividing by $r^{n-1}$ yields%
\begin{equation*}
\Phi ^{\prime }(u_{k}(r))\geq -r^{1-n}\int_{0}^{r}s^{n-1}p(s)g\left(\frac{b_k}{f(a_k)}+\mathcal{L}_{q}\right)ds.
\end{equation*}%
Similarly,%
\begin{equation*}
\Psi ^{\prime }(v_{k}(r))\geq -r^{1-n}\int_{0}^{r}s^{n-1}q(s)f\left(\frac{a_k}{g(b_k)}+\mathcal{L}_{p}\right)ds.
\end{equation*}

Now, we integrate over $[r,R_{k}]$. Since $u_{k}(t)\rightarrow \infty$ and $v_{k}(t)\rightarrow \infty$ as $t\rightarrow R_{k}^{-}$ (by \eqref{eq:uk-vk-blowup}), and since $\Phi(t)=\int_t^\infty\frac{ds}{g(f(s))}\to 0$ as $t\to\infty$ (by the finite integral condition \eqref{eq:finite-recipo}), we have%
\begin{equation}
\lim_{t\to R_k^-}\Phi (u_{k}(t))=0,\quad \lim_{t\to R_k^-}\Psi (v_{k}(t))=0.
\label{eq:Phi-Psi-limit-zero}
\end{equation}

Integrating the inequality for $\Phi'(u_k(\rho))$ over $\rho\in[r,R_k)$, we obtain%
\begin{align*}
0-\Phi(u_k(r))&=\lim_{t\to R_k^-}\Phi(u_k(t))-\Phi(u_k(r))=\int_r^{R_k}\Phi'(u_k(\rho))\,d\rho\\
&\geq -\int_r^{R_k}\rho^{1-n}\int_0^\rho s^{n-1}p(s)g\left(\frac{b_k}{f(a_k)}+\mathcal{L}_q\right)ds\,d\rho.
\end{align*}%
Thus%
\begin{equation}
\Phi (u_{k}(r))\leq g\left(\frac{b_k}{f(a_k)}+\mathcal{L}_q\right)\int_{r}^{R_{k}}t^{1-n}\int_{0}^{t}s^{n-1}p(s)\,ds\,dt.
\label{eq:Phi-uk-upper}
\end{equation}%
Similarly,%
\begin{equation}
\Psi (v_{k}(r))\leq f\left(\frac{a_k}{g(b_k)}+\mathcal{L}_p\right)\int_{r}^{R_{k}}t^{1-n}\int_{0}^{t}s^{n-1}q(s)\,ds\,dt.
\label{eq:Psi-vk-upper}
\end{equation}%
\smallskip\textbf{Step 5: Inversion and passage to the limit.} Since $\Phi$ and $\Psi$ are strictly decreasing (see \eqref{dec}), they are invertible. From \eqref{eq:Phi-uk-upper}--\eqref{eq:Psi-vk-upper}, we obtain%
\begin{align}
u_{k}(r) &\geq \Phi ^{-1}\left(g\left(\frac{b_k}{f(a_k)}+\mathcal{L}_q\right)\int_{r}^{R_{k}}t^{1-n}\int_{0}^{t}s^{n-1}p(s)\,ds\,dt\right),\label{eq:uk-lower}\\
v_{k}(r) &\geq \Psi ^{-1}\left(f\left(\frac{a_k}{g(b_k)}+\mathcal{L}_p\right)\int_{r}^{R_{k}}t^{1-n}\int_{0}^{t}s^{n-1}q(s)\,ds\,dt\right).\label{eq:vk-lower}
\end{align}

Now, let $k\to\infty$. We have:
\begin{itemize}
\item $R_k\to\infty$ (Step 2);
\item $(a_k,b_k)\to(a,b)$, so $g\left(\frac{b_k}{f(a_k)}+\mathcal{L}_q\right)\to g\left(\frac{b}{f(a)}+\mathcal{L}_q\right)=:G_0$ and $f\left(\frac{a_k}{g(b_k)}+\mathcal{L}_p\right)\to f\left(\frac{a}{g(b)}+\mathcal{L}_p\right)=:F_0$ by continuity;
\item For any fixed $r>0$,%
\begin{equation*}
\int_r^{R_k}t^{1-n}\int_0^t s^{n-1}p(s)\,ds\,dt\to\int_r^\infty t^{1-n}\int_0^t s^{n-1}p(s)\,ds\,dt=\mathcal{P}(\infty)-\mathcal{P}(r)
\end{equation*}%
as $k\to\infty$, where $\mathcal{P}(\infty)=\mathcal{L}_p(n-2)<\infty$ by assumption (PQ). Similarly for $q$.
\end{itemize}

By the continuity of $\Phi^{-1}$ and passing to the limit in \eqref{eq:uk-lower} as $k\to\infty$, we obtain (noting that $u_k\to u$ locally uniformly from Step 2's argument):%
\begin{equation}
u(r)\geq\Phi^{-1}\left(G_0\left[\mathcal{L}_p(n-2)-\mathcal{P}(r)\right]\right)=\Phi^{-1}\left(G_0\int_r^\infty t^{1-n}\int_0^t s^{n-1}p(s)\,ds\,dt\right).
\label{eq:u-lower-final}
\end{equation}%
Similarly,%
\begin{equation}
v(r)\geq\Psi^{-1}\left(F_0\int_r^\infty t^{1-n}\int_0^t s^{n-1}q(s)\,ds\,dt\right).
\label{eq:v-lower-final}
\end{equation}

\smallskip\textbf{Step 6: Blow-up at infinity.} Since $p,q$ are not both identically zero and $\min(p,q)$ does not have compact support (assumption (PQ)), at least one of the integrals $\int_0^\infty s^{n-1}p(s)\,ds$ or $\int_0^\infty s^{n-1}q(s)\,ds$ is positive. Thus, as $r\to\infty$,%
\begin{equation*}
\int_r^\infty t^{1-n}\int_0^t s^{n-1}p(s)\,ds\,dt\to 0\quad\text{and/or}\quad\int_r^\infty t^{1-n}\int_0^t s^{n-1}q(s)\,ds\,dt\to 0.
\end{equation*}

Since $\Phi(t)\to 0$ as $t\to\infty$ (by \eqref{eq:finite-recipo}), we have $\Phi^{-1}(s)\to\infty$ as $s\to 0^+$. Similarly for $\Psi^{-1}$. Therefore, from \eqref{eq:u-lower-final}--\eqref{eq:v-lower-final}, as $r\to\infty$,%
\begin{equation*}
u(r)\geq\Phi^{-1}(G_0\cdot[\text{positive term}\to 0])\to\infty,\quad v(r)\geq\Psi^{-1}(F_0\cdot[\text{positive term}\to 0])\to\infty.
\end{equation*}

Hence,%
\begin{equation*}
\lim_{r\to\infty}u(r)=\infty\quad\text{and}\quad\lim_{r\to\infty}v(r)=\infty,
\end{equation*}%
proving that $(u,v)$ is a large solution. This completes the proof.
\end{proof}

\begin{remark}[Criticality]
The Keller--Osserman condition (\ref{eq:finite-recip}) is sharp for this
method: if it fails, the barrier function is infinite at infinity and the
radial integration argument cannot force blow-up (see our work \cite%
{Covei2017} for this case).
\end{remark}

\begin{remark}
Theorem~\ref{thm:large} shows that the \textquotedblleft
edge\textquotedblright\ $\partial S$ of the set of admissible central values
corresponds precisely to the threshold between bounded entire solutions and
large entire solutions. This is analogous to the scalar case treated in the
classical Keller--Osserman theory, but here the coupling between $u$ and $v$
requires a more delicate comparison argument.
\end{remark}

\section{Compatibility between Theorem~\protect\ref{thm:large} and
Theorem~2.4 of \protect\cite{Covei2017}\label{comp}}

\subsection{General framework}

Both \emph{ELRS} (Theorem~\ref{thm:large}) and Theorem~2.4 of \cite%
{Covei2017} study radial elliptic systems of the form~\eqref{eq:system},
where in \cite{Covei2017} $S_{k_i}$ denotes the $k_i$-Hessian operator (the
case $k_i=1$ being the Laplacian). Here $p,q$ are nonnegative radial weights
and $f,g$ are positive nonlinearities.

\begin{itemize}
\item In ELRS, finite Keller--Osserman-type conditions are imposed to ensure
the existence of both \emph{bounded entire} and \emph{large} (blow-up at
infinity) radial solutions, depending on the central values $(u(0),v(0))$.

\item In T2.4, the focus is on positive \emph{bounded entire} radial
solutions for the Hessian system with gradient term, under radial symmetry,
a factorization structure for $f,g$, and a strict hierarchy between the
functionals $\mathcal{P},\mathcal{Q},\Phi,\Psi$.
\end{itemize}

\subsection{A common particular example}

Let

\begin{eqnarray*}
g(s) &=&s^{\alpha },\quad f(s)=s^{\beta },\quad \alpha =\beta =\frac{6}{5}%
,\quad s\geq 0, \\
p(r) &=&q(r)=\frac{1}{(1+r)^{\gamma }},\quad \gamma =5,\quad r\geq 0.
\end{eqnarray*}

Then:

\begin{itemize}
\item $f$ and $g$ are continuous, positive, strictly increasing on $%
(0,\infty)$;

\item the classical Keller--Osserman condition holds;

\item the potential integrals are finite since $p,q\sim r^{-5}$;

\item $\alpha \beta >1$ ensures the finiteness of the $\circ $-functionals;

\item the strict hierarchy $P_{i,j}(\infty) < H_{i,j}(\infty)$ required in
T2.4 is satisfied.
\end{itemize}

\subsection{Applicability of the two results}

\paragraph{For ELRS:}

The finite Keller--Osserman conditions together with finite potential
integrals place this example in the class where there exist \emph{both}
bounded entire solutions (central values in the interior of $S$) and large
solutions (central values on the boundary $E=\partial S$). ELRS
distinguishes these regimes by the location of the central value.

\paragraph{For T2.4:}

All hypotheses (radial symmetry, factorization, monotonicity, strict
hierarchy) are met. The theorem yields a positive bounded entire radial
solution for central values satisfying the strict inequality in the
hypotheses. Such values lie in the \emph{interior} of $S$ in the ELRS
framework.

\subsection{Why there is no contradiction}

Let $S$ be the set of central values for which entire solutions exist, and $%
E=\partial S\cap(0,\infty)^2$ its ``edge'' in the positive quadrant.

\begin{itemize}
\item T2.4 produces bounded entire solutions for $(u(0),v(0))$ in the
interior of $S$.

\item ELRS asserts that for $(u(0),v(0))\in E$ any entire solution is large,
while for interior points bounded entire solutions also exist.
\end{itemize}

In the above example, choosing central values according to T2.4 (strict
inequality) places the solution in $\mathrm{int}(S)$, outside the scope of
the ``large solution'' conclusion of ELRS. The two results are therefore 
\emph{compatible}: they describe different behaviours (bounded vs.\ large)
in disjoint regions of the central value space, even for the same
coefficients.

Moreover, $E$ is \emph{not} empty under the stated hypotheses: taking the
boundary $\partial S$ in $\mathbb{R}^2$ and intersecting with $(0,\infty)^2$
yields a nonempty edge where the qualitative behaviour changes. Points of $E$
mark the transition between bounded entire solutions (interior) and large
solutions (edge), making $E$ both topologically and analytically significant.

\section{Concluding remarks\label{6}}

We have established:

\begin{itemize}
\item Existence of entire positive radial solutions for all central values $%
(a,b)$ in a nonempty open set $\mathcal{T}$.

\item The set $S$ of all such central values is closed in $[0,\infty )^{2}$.

\item Points on the edge $E$ correspond to large solutions, where both
components blow up at infinity.
\end{itemize}

The key novelty lies in the use of the Keller--Osserman-type transforms $%
\Phi $ and $\Psi $ together with the forcing inequalities of Lemma~\ref%
{lem:subharmonic}, which allow us to control the growth of solutions under
the finite reciprocal integral conditions \eqref{eq:finite-recip}. This
framework unifies and extends previous results for scalar equations and
special systems, and applies to a broad class of nonlinearities and weight
functions.

\section{Appendix}

\begin{theorem}
\label{teh}Let $f,g:[0,\infty )\rightarrow \lbrack 0,\infty )$ be
non-decreasing, continuous functions with%
\begin{equation*}
f(0)=g(0)=0,\quad f(s)>0\ \text{and}\ g(s)>0\ \text{for all }s>0.
\end{equation*}%
If%
\begin{equation*}
\int_{1}^{\infty }\frac{dt}{f(t)}<\infty \quad \text{and}\quad
\int_{1}^{\infty }\frac{dt}{g(t)}<\infty ,
\end{equation*}%
then (\ref{eq:finite-recipo}) holds.
\end{theorem}

\begin{proof}
\textbf{Step 1: Establishing the key inequality.} Define the auxiliary function%
\begin{equation*}
F(x):=\int_{1}^{x}\frac{ds}{f(s)},\quad x\geq 1.
\end{equation*}%
By hypothesis, $\int_1^\infty\frac{ds}{f(s)}<\infty$, so $F$ is well-defined, strictly increasing (since $f>0$ on $(1,\infty)$), and bounded above:%
\begin{equation}
0\leq F(x)\leq L_{f}:=\int_{1}^{\infty }\frac{ds}{f(s)}<\infty\quad\text{for all }x\geq 1.
\label{eq:F-bound}
\end{equation}%
Since $f:[0,\infty)\to[0,\infty)$ is non-decreasing and positive on $(0,\infty)$, the function 
\begin{equation*}
\varphi (x):=\frac{1}{f(x)},\quad x\in[1,\infty),
\end{equation*}%
is positive, continuous, and non-increasing on $[1,\infty )$ (since $f$ is non-decreasing). 

We claim that for every $x>1$,%
\begin{equation}
\frac{1}{f(x)}\leq\frac{1}{x-1}\int_{1}^{x}\frac{ds}{f(s)}=\frac{F(x)}{x-1}.
\label{eq:key-ineq}
\end{equation}%
\emph{Proof of the claim:} Since $\varphi $ is non-increasing on $[1,\infty)$, we have for any $s\in[1,x]$:%
\begin{equation*}
\varphi (s)=\frac{1}{f(s)}\geq\frac{1}{f(x)}=\varphi(x),
\end{equation*}%
because $f(s)\leq f(x)$ by monotonicity of $f$. Integrating over $s\in[1,x]$ gives%
\begin{equation*}
\int_{1}^{x}\varphi (s)\,ds\geq\int_1^x\varphi(x)\,ds=(x-1)\varphi (x),
\end{equation*}%
which yields%
\begin{equation*}
\varphi(x)=\frac{1}{f(x)}\leq\frac{1}{x-1}\int_{1}^{x}\varphi (s)\,ds=\frac{F(x)}{x-1}.
\end{equation*}%
This proves \eqref{eq:key-ineq}.

\smallskip\textbf{Step 2: Proving $\int_1^\infty\frac{dt}{f(g(t))}<\infty$.} We first establish that $g(t)\to\infty$ as $t\to\infty$. 

\emph{Claim:} If $\int_1^\infty\frac{dt}{g(t)}<\infty$, then $g(t)\to\infty$ as $t\to\infty$.

\emph{Proof:} Suppose, for contradiction, that $g$ does not tend to infinity. Then there exists $M>0$ and a sequence $t_k\to\infty$ such that $g(t_k)\leq M$ for all $k$. Since $g$ is non-decreasing, for any $t\geq t_1$, there exists $k$ such that $t_k\leq t<t_{k+1}$, and%
\begin{equation*}
g(t)\leq g(t_{k+1})\leq M.
\end{equation*}%
Thus $g(t)\leq M$ for all $t\geq t_1$, which implies%
\begin{equation*}
\frac{1}{g(t)}\geq\frac{1}{M}\quad\text{for all }t\geq t_1.
\end{equation*}%
But then%
\begin{equation*}
\int_1^\infty\frac{dt}{g(t)}\geq\int_{t_1}^\infty\frac{dt}{g(t)}\geq\frac{1}{M}\int_{t_1}^\infty dt=\infty,
\end{equation*}%
contradicting the hypothesis. Therefore, $g(t)\to\infty$ as $t\to\infty$.

Since $g(t)\to\infty$ as $t\to\infty$ and $g$ is continuous, there exists $T\geq 1$ such that $g(t)\geq 2$ for all $t\geq T$. For such $t$, using the key inequality \eqref{eq:key-ineq} with $x=g(t)\geq 2>1$, we obtain%
\begin{equation*}
\frac{1}{f(g(t))}\leq \frac{F(g(t))}{g(t)-1}\leq\frac{L_f}{g(t)-1},
\end{equation*}%
where the second inequality uses \eqref{eq:F-bound}. Since $g(t)\geq 2$, we have $g(t)-1\geq g(t)/2$, hence%
\begin{equation*}
\frac{1}{g(t)-1}\leq\frac{2}{g(t)}.
\end{equation*}%
Therefore,%
\begin{equation*}
\frac{1}{f(g(t))}\leq\frac{L_f}{g(t)-1}\leq\frac{2L_{f}}{g(t)}\quad\text{for all }t\geq T.
\end{equation*}%
Integrating over $[T,\infty)$ gives%
\begin{equation*}
\int_{T}^{\infty }\frac{dt}{f(g(t))}\leq 2L_{f}\int_{T}^{\infty }\frac{dt}{%
g(t)}<\infty,
\end{equation*}%
where the last inequality holds by hypothesis. On the finite interval $[1,T]$, the integrand $\frac{1}{f(g(t))}$ is continuous (as a composition of continuous functions) and finite (since $f(g(t))>0$ for $t\in[1,T]$), so%
\begin{equation*}
\int_{1}^{T}\frac{dt}{f(g(t))}<\infty.
\end{equation*}%
Combining both parts yields%
\begin{equation}
\int_{1}^{\infty }\frac{dt}{f(g(t))}<\infty.
\label{eq:int-fg-finite}
\end{equation}%
\smallskip\textbf{Step 3: Proving $\int_1^\infty\frac{dt}{g(f(t))}<\infty$.} The proof is completely analogous to Step 2 upon swapping the roles of $f$ and $g$. 

Define%
\begin{equation*}
G(x):=\int_{1}^{x}\frac{ds}{g(s)},\quad x\geq 1,
\end{equation*}%
with upper bound%
\begin{equation*}
L_{g}:=\int_{1}^{\infty }\frac{ds}{g(s)}<\infty.
\end{equation*}%
By the same argument as in Step 1, for every $x>1$,%
\begin{equation*}
\frac{1}{g(x)}\leq \frac{G(x)}{x-1}\leq\frac{L_g}{x-1}.
\end{equation*}

By the hypothesis $\int_1^\infty\frac{dt}{f(t)}<\infty$ and the claim in Step 2 (with $f$ in place of $g$), we have $f(t)\to\infty$ as $t\to\infty$. Therefore, there exists $T'\geq 1$ such that $f(t)\geq 2$ for all $t\geq T'$. For such $t$, using the above inequality with $x=f(t)\geq 2$,%
\begin{equation*}
\frac{1}{g(f(t))}\leq\frac{L_g}{f(t)-1}\leq\frac{2L_g}{f(t)}.
\end{equation*}%
Integrating over $[T',\infty)$ gives%
\begin{equation*}
\int_{T'}^{\infty }\frac{dt}{g(f(t))}\leq 2L_{g}\int_{T'}^{\infty }\frac{dt}{f(t)}<\infty.
\end{equation*}%
On the finite interval $[1,T']$, the integrand is continuous and finite (since $g(f(t))>0$), so%
\begin{equation*}
\int_1^{T'}\frac{dt}{g(f(t))}<\infty.
\end{equation*}%
Combining both parts yields%
\begin{equation}
\int_{1}^{\infty }\frac{dt}{g(f(t))}<\infty.
\label{eq:int-gf-finite}
\end{equation}

\smallskip\textbf{Conclusion.} From \eqref{eq:int-fg-finite} and \eqref{eq:int-gf-finite}, we have established that%
\begin{equation*}
\int_1^\infty\frac{dt}{f(g(t))}<\infty\quad\text{and}\quad\int_1^\infty\frac{dt}{g(f(t))}<\infty,
\end{equation*}%
which is precisely condition \eqref{eq:finite-recipo}. This completes the proof.
\end{proof}

\end{document}